\documentclass[11pt]{article}
\usepackage{amsmath,amsfonts,amssymb, amsthm,enumerate,graphicx}
\usepackage{tikz,authblk}
\usepackage{subfig}
\usepackage{url}
\newtheorem{theorem} {{\textsf{Theorem}}}
\newtheorem{proposition}[theorem]{{\textsf{Proposition}}}
\newtheorem{corollary}[theorem]{{\textsf{Corollary}}}

\newtheorem{definition}[theorem]{{\textsf{Definition}}}
\newtheorem{remark}[theorem]{{\textsf{Remark}}}

\newtheorem{lemma}[theorem]{{\textsf{Lemma}}}

\begin{document}
\title{Genus-minimal crystallizations of PL 4-manifolds}

\author{Biplab Basak}
\date{}
\maketitle
\vspace{-10mm}
\begin{center}

\noindent {\small Department of Mathematics, Cornell University, Ithaca, NY 14853, USA.}

\noindent {\small {\em E-mail address:} \url{bb564@cornell.edu}}

\medskip

\date{February 10, 2018}
\end{center}
\hrule

\begin{abstract}
For $d\geq 2$, the regular genus of a closed connected PL $d$-manifold $M$ is the least genus (resp., half of genus) of an
orientable (resp., a non-orientable) surface into which a crystallization of $M$ imbeds regularly. The regular genus of every orientable surface equals its genus, and the regular genus of every 3-manifold equals its Heegaard genus. For every closed connected PL $4$-manifold $M$, it is known that its regular genus $\mathcal G(M)$ is at least $2 \chi (M) + 5m -4$, where $m$ is the rank of the fundamental group of $M$. In this article, we introduce the concept of ``weak semi-simple crystallization" for every closed connected PL $4$-manifold $M$, and prove that  $\mathcal G(M)=  2 \chi (M) + 5m -4$  if and only if $M$ admits a weak semi-simple crystallization. We then show that the PL invariant regular genus is additive under  the connected sum within the class of all PL 4-manifolds admitting a weak semi-simple crystallization. Also, we note that this property is related to 
the 4-dimensional Smooth Poincar\'e Conjecture.
\end{abstract}

\noindent {\small {\em MSC 2010\,:} Primary 57Q15. Secondary 57Q05; 57N13; 05C15.


}

\medskip

\hrule
\section{Introduction and Main Results}
First, let us recall that a crystallization of a PL-manifold is a contracted colored graph which represents the manifold (for details and related notations see Subsection \ref{crystal}). The existence of crystallizations for every closed connected PL-manifold is ensured by a classical theorem due to Pezzana (see \cite{pe74}, or \cite{fgg86} for subsequent generalizations).

Let $(\Gamma, \gamma)$ be a $(d+1)$-regular colored graph (i.e., through each vertex of the graph, there are  $d+1$ edges with different colors from the color set $\Delta_d:=\{0, \dots, d\}$). An embedding $i : \Gamma \hookrightarrow F$ of $\Gamma$ into a closed surface $F$ is called {\it regular} if there exists a cyclic permutation $\varepsilon=(\varepsilon_0, \varepsilon_1, \dots, \varepsilon_d)$ of the color set $\Delta_d$, such that the boundary of each face of $i(\Gamma)$ is a bi-colored cycle with edges alternately colored by $\varepsilon_j, \varepsilon_{j+1}$ for some $j$ (the addition is modulo $d+1$). Then,
the {\it regular genus} $\rho (\Gamma)$ of $(\Gamma,\gamma)$ is the least genus (resp., half of genus) of an
orientable (resp., a non-orientable) surface into which $\Gamma$ embeds regularly, and the {\it regular genus} $\mathcal G (M)$ of a closed connected PL $d$-manifold $M$ is defined as the minimum regular genus of its crystallizations. Note that the notion of regular genus extends classical notions to arbitrary dimension. In fact, the regular genus of a closed connected orientable (resp., non-orientable) surface coincides with its genus (resp., half of its genus), while the regular genus of a closed connected 3-manifold coincides with its Heegaard genus (see  \cite{ga81}).

The invariant regular genus  has been intensively studied, yielding some important general results. For example, regular genus zero characterizes the $d$-sphere among all closed connected PL $d$-manifolds (cf. \cite{[FG$_2$]}). 
In particular, in dimension $d \in \{4,5\}$, a lot of classifying results in PL-category have been obtained, both for closed and bounded PL $d$-manifolds (see, for example, \cite{[C],{Casali},[CG]}). In fact, many authors give  bounds for regular genus of some PL $d$-manifolds (cf. \cite{cs7,cr4,fg82,gg93,sp99}).

Let $(\Gamma,\gamma)$ be a crystallization (with color set $\Delta_4$) of a closed connected PL $4$-manifold $M$ and let $m$ be the rank (the minimum size of a generating set) of the fundamental group of $M$.
Let $g_{\{i,j,k\}}$ be the number of connected components of the subgraph of $(\Gamma,\gamma)$ restricted in the color set $\{i,j,k\}$. Then, by a result of Gagliardi (cf. Proposition \ref{prop:preliminaries} $(b)$), $g_{\{i,j,k\}} \geq m+1$ for any distinct $i,j,k \in \Delta_4$. A crystallization $(\Gamma,\gamma)$ of $M$ is called a {\em semi-simple crystallization} if $g_{\{i,j,k\}}=m+1$ for all $i,j,k \in \Delta_4$. Now, generalizing the notion of semi-simple crystallization, we have defined the following.

\begin{definition}\label{def:weak}
{\rm Let $M$ be a closed connected PL 4-manifold and $m$ be the rank of the fundamental group of $M$. A crystallization $(\Gamma,\gamma)$ of $M$ with color set $\Delta_4$ is said to be {\em a weak semi-simple crystallization} if $g_{\{i,i+1,i+2\}}=m+1$ for $i \in \Delta_4$ (addition in subscript of $g$ is modulo 5).}
\end{definition} 

From \cite{bc15} we know that, for a closed connected PL $4$-manifold $M$, its regular genus $ \mathcal G(M) \ \ge \ 2 \chi (M) + 5m -4$, and equality holds if $M$ admits a semi-simple crystallization, where $m$ is the rank of the fundamental group of $M$. Here, we have proved the following. 

\begin{theorem}\label{teorem:main}
Let $M$ be a closed connected PL $4$-manifold and $m$ be the rank of its fundamental group. Then $\mathcal G(M)=  2 \chi (M) + 5m -4$  if and only if $M$ admits a weak semi-simple crystallization. 
\end{theorem}

We have also shown that the class of PL 4-manifolds admitting weak semi-simple crystallizations is closed under connected sum (cf. Lemma \ref{Theorem:additive}). Hence, PL $4$-manifolds admitting weak semi-simple crystallizations actually constitutes a huge class  which comprehends  $S^4$, $\mathbb{CP}^{2}$, $S^{2} \times S^{2}$,  $\mathbb{RP}^4$, $\mathbb{RP}^2\times S^2$,  the orientable and non-orientable $S^3$-bundles over $S^1$, an orientable and a non-orientable $(S^2\times S^1)$-bundles over $S^1$, and the $K3$-surface, together with their connected sums, possibly by taking copies with reversed orientation (cf. Remark \ref{rem:huge-class}).

The additivity of regular genus under connected sum has been conjectured for closed connected PL $d$-manifolds. Moreover, the associated (open) problem is significant especially in dimension four. In fact, in dimension four, additivity of regular genus, at least in the simply-connected case, would imply the 4-dimensional Smooth Poincar\'e Conjecture, in virtue of a well-known Wall's Theorem (cf. \cite{[W]}). Here, we have proved the additivity of  regular genus for the class of PL $4$-manifolds admitting weak semi-simple crystallizations.

\begin{theorem}\label{theorem:addative}
Let $M_1$ and $M_2$ be two PL $4$-manifolds admitting weak semi-simple crystallizations. Then $\mathcal{G}(M_1\# M_2) = \mathcal{G}(M_1)+\mathcal{G}(M_2)$.
\end{theorem} 
We have also shown some interesting consequences of the above theorems (cf. Corollaries \ref{cor:min-genus}, \ref{corollary:novik-swartz} and \ref{corollary:genus}, and Remarks \ref{rem:huge-class} and \ref{rem:conjecture}).
\section{Preliminaries}

\subsection{Crystallizations} \label{crystal}
A multigraph $\Gamma = (V(\Gamma), E(\Gamma))$ is a finite connected graph, with vertex-set $V(\Gamma)$ and edge-set $E(\Gamma)$,  which can have multiple edges but no loops. A multigraph $\Gamma$ is called {\em $(d+1)$-regular} if the number of edges adjacent to each vertex is $(d+1)$. For standard terminology on graphs, we refer to \cite{bm08}. A {\it $(d+1)$-colored graph} is a pair
$(\Gamma,\gamma)$, where $\Gamma= (V(\Gamma),$ $E(\Gamma))$ is a multigraph of degree at most $d+1$ and the surjective map $\gamma : E(\Gamma) \to \Delta_d:=\{0,1, \dots , d\}$ is a proper edge-coloring (i.e., $\gamma(e) \ne \gamma(f)$ for any pair $e,f$ of adjacent edges). The elements of the set $\Delta_d$ are called the {\it colors} of
$\Gamma$.

For each $B \subseteq \Delta_d$ with $h$ elements, the
graph $\Gamma_B =(V(\Gamma), \gamma^{-1}(B))$ is an $h$-colored graph with edge-coloring $\gamma|_{\gamma^{-1}(B)}$.
If $\Gamma_{\Delta_d \setminus\{c\}}$ is connected for all $c\in \Delta_d$, then  $(\Gamma,\gamma)$ is called {\em contracted}. For a color set $\{i_1,i_2,\dots,i_k\} \subset \Delta_d$, $g_{\{i_1,i_2, \dots, i_k\}}$ denotes the number of connected components of the graph $\Gamma_{\{i_1, i_2, \dots, i_k\}}$.

Each $(d+1)$-colored graph uniquely determines a $d$-dimensional simplicial cell-complex ${\mathcal K}(\Gamma)$, which is said to be {\it associated to $\Gamma$}:
\begin{itemize}
\item{} for every vertex $v\in V(\Gamma)$, take a $d$-simplex $\sigma(v)$ and label injectively its $d+1$ vertices by the colors of $\Delta_d$;
\item{} for every  edge of color $i$ between $v,w\in V(\Gamma)$, identify the ($d-1$)-faces of $\sigma(v)$ and $\sigma(w)$ opposite to $i$-labelled vertices, so that equally labelled vertices coincide.
\end{itemize}
The vector $f({\mathcal K}(\Gamma)) := (f_0 ({\mathcal K}(\Gamma)), f_1 ({\mathcal K}(\Gamma)), \ldots , f_d ({\mathcal K}(\Gamma)))$ is called the {\em $f$-vector} of ${\mathcal K}(\Gamma)$ where $f_i({\mathcal K}(\Gamma))$ is the number of $i$-simplices in ${\mathcal K}(\Gamma)$. If the geometrical carrier of ${\mathcal K}(\Gamma)$ is PL-homeomorphic to a PL $d$-manifold $M$, then the $(d+1)$-colored graph $(\Gamma,\gamma)$ is said to {\it represent} $M$; if, in addition, $(\Gamma,\gamma)$ is contracted, then it is called a {\it crystallization} of $M$. If $(\Gamma, \gamma)$ is a crystallization of a connected PL $d$-manifold $M$, then the number of vertices in ${\mathcal K}(\Gamma)$ is $d+1$. If $(\Gamma,\gamma)$ is a crystallization of a closed connected PL $d$-manifold then $(\Gamma,\gamma)$ is a $(d+1)$-regular colored graph, i.e., a $(d+1)$-colored graph which is also $(d+1)$-regular.

Let $(\Gamma_1,\gamma_1)$ and $(\Gamma_2,\gamma_2)$ be two disjoint $(d+1)$-colored graphs with the same color set $\Delta_d$, and let $v_i \in V_i$ ($1\leq i\leq 2$).
The {\em connected sum} of $\Gamma_1$, $\Gamma_2$  with respect to vertices $v_1, v_2$ (denoted by $(\Gamma_1\#_{v_1v_2}\Gamma_2, \gamma_1 \#\gamma_2)$, or simply $(\Gamma_1\#\Gamma_2, \gamma_1 \#\gamma_2)$) is the
graph obtained from $(\Gamma_1 \setminus\{v_1\}) \sqcup (\Gamma_2 \setminus \{v_2\})$ by adding $d+1$ new edges $e_0,
\dots, e_d$ with colors $0, \dots, d$ respectively, such that the end points of $e_j$ are $u_{j,1}$ and $u_{j,
2}$, where $v_i$ and $u_{j,i}$ are joined in $(\Gamma_i,\gamma_i)$ with an edge of color $j$ for $0\leq j\leq d$, $1\leq
i\leq 2$. From \cite{fgg86}, we know the following.

\begin{proposition} \label{prop:preliminaries}
For $d\geq 3$, let  $(\Gamma,\gamma)$ be a crystallization of a PL $d$-manifold $M$.
\begin{itemize}
\item[(a)] Let $\nu(\Gamma)$ denote the number of vertices of $\Gamma$. Then,   $2g_{\{i,j,k\}}=g_{\{i,j\}}+g_{\{i,k\}}+g_{\{j,k\}}-\frac{\nu(\Gamma)}{2}$ for any distinct $i,j,k \in \Delta_d$.
\item[(b)] For any distinct $i, j \in \Delta_d$, the set of all connected components of  $\Gamma_{\Delta_d \setminus \{i,j\}}$, but one, is in bijection with a set of generators of the fundamental group $\pi_1(M).$
\item[(c)] If $(\Gamma^{\prime},\gamma^{\prime})$ is a crystallization of a PL $d$-manifold $M^{\prime},$ then the graph connected sum $(\Gamma \# \Gamma^{\prime}, \gamma \# \gamma^{\prime})$ is a crystallization of $M \# M^{\prime}$.
 \end{itemize}
\end{proposition}

\subsection{Regular Genus of PL-manifolds}\label{sec:genus}
The notion of {\it regular genus} is strictly related to the existence of {\it regular embeddings} of crystallizations into closed surfaces, i.e., embeddings whose regions are bounded by the images of bi-colored cycles, with colors consecutive in a fixed permutation of the color set.
More precisely, according to \cite{ga81},  if $(\Gamma, \gamma)$ is a crystallization of an orientable (resp., non-orientable) PL $d$-manifold $M$ ($d \geq 3$), then for each cyclic permutation  $\varepsilon= (\varepsilon_0, \varepsilon_1, \varepsilon_2, \dots , \varepsilon_d)$ of $\Delta_d$,  a regular embedding $i_\varepsilon : \Gamma \hookrightarrow F_\varepsilon$ exists,  where $F_{\varepsilon}$ is the closed orientable (resp., non-orientable) surface with Euler characteristic
$$
\chi_{\varepsilon}(\Gamma)= \sum_{i \in \mathbb{Z}_{d+1}}g_{\{\varepsilon_i,\varepsilon_{i+1}\}} + (1-d) \ \frac{\nu(\Gamma)}{2}.$$

In the orientable (resp., non-orientable) case, the integer
$$\rho_{\varepsilon}(\Gamma) = 1 - \chi_{\varepsilon}(\Gamma)/2$$
is equal to the genus (resp., half of the genus) of the surface $F_{\varepsilon}$.
Then, the regular genus $\rho (\Gamma)$ of $(\Gamma,\gamma)$  and the regular genus $\mathcal G (M)$ of $M$ are:
$$\rho(\Gamma)= \min \{\rho_{\varepsilon}(\Gamma) \ | \  \varepsilon \ \text{ is a cyclic permutation of } \ \Delta_d\};$$
$$\mathcal G(M) = \min \{\rho(\Gamma) \ | \  (\Gamma,\gamma) \mbox{ is a crystallization of } M\}.$$

From \cite{bc15}, we know the following.
\begin{proposition} \label{prop:lowerbound}
Let $M$ be a (closed connected) PL $4$-manifold and $m$ be the rank of its fundamental group. Then, $ \mathcal G(M) \ \ge \ 2 \chi (M) + 5m -4.$
\end{proposition}

\section{Proof of Main Results and some Consequences}

\begin{lemma}\label{lemma:vertex}
Let $(\Gamma,\gamma)$ be a crystallization of a closed connected PL $4$-manifold $M$. Then $\nu(\Gamma) = 6 \chi(M)+2\sum_{0 \leq i <j<k\leq 4} g_{\{i,j,k\}}-30.$
\end{lemma}

\begin{proof}
Let $2p$ be the number of vertices of $\Gamma$ and $X=\mathcal{K}(\Gamma)$ be the corresponding simplicial cell-complex. Then the Dehn-Sommerville equations in dimension four yield:
	$$
	\begin{array}{lllllllllll}
		f_0 (X) &-& f_1(X) &+& f_2(X) &-& f_3(X) &+& f_4(X) &=& \chi(M), \\
		&& 2 f_1(X) &-& 3 f_2(X) &+& 4 f_3(X) &-& 5 f_4(X) &=& 0, \\
		&&&&&& 2 f_3(X) &-& 5 f_4(X) &=& 0. \\
	\end{array}
	$$
	
\noindent Since $f_0 (X) = 5$ (by construction) and $f_4(X)=\nu(\Gamma)=2p$, the following equality holds:

$$2p = 6 \chi(M)+2f_1(\mathcal{K}(\Gamma))-30.$$

\noindent Since $f_1(\mathcal{K}(\Gamma)) = \sum_{0 \leq i <j<k\leq 4} g_{\{i,j,k\}}$, we have $2p = 6 \chi(M)+2\sum_{0 \leq i <j<k\leq 4} g_{\{i,j,k\}}-30.$
\end{proof}

\begin{lemma}\label{lemma:main1}
Let $(\Gamma,\gamma)$ be a crystallization of a closed connected PL $4$-manifold $M$
and $m$ be the rank of the fundamental group of $M$. Then, for any cyclic permutation $\varepsilon=(\varepsilon_0,\dots,\varepsilon_4=4)$ of the color set $\Delta_4$ we have $\rho_{\varepsilon} (\Gamma) = 2 \chi (M) + 5m -4 +\sum_{i \in \mathbb{Z}_5} (g_{\{\varepsilon_i,\varepsilon_{i+2},\varepsilon_{i+4}\}}-m-1)$.
\end{lemma}

\begin{proof}
From Proposition \ref{prop:preliminaries} $(b)$, we know that $g_{\{i,j,k\}} \geq m+1$ for any distinct $i,j,k \in \Delta_4$. Therefore, let us assume that $g_{\{i,j,k\}}=(m+1)+ t_{\{i,j,k\}}$, for $t_{\{i,j,k\}}  \in \mathbb{Z}$ and $t_{\{i,j,k\}} \geq 0$. Thus, from Lemma \ref{lemma:vertex}, we have $\nu(\Gamma) = 6 \chi(M)+2\sum_{0 \leq i <j<k\leq 4} g_{\{i,j,k\}}-30= 6 \chi(M)+20(m+1)-30 +2\sum_{0 \leq i <j<k\leq 4} t_{\{i,j,k\}}$. We set $\nu(\Gamma) = 2\bar{p} + 2q$, where $2\bar{p}=6 \chi(M)+20(m+1)-30$ and $q = \sum_{0 \leq i <j<k\leq 4} t_{\{i,j,k\}}$. Therefore, from the definition of the regular genus, we can deduce that

\begin{equation}\label{eq:rho}
\rho_{\varepsilon} (\Gamma) = 1 + \frac{3}{4} \nu(\Gamma) - \frac{1}{2}\sum \limits_{i\in \mathbb{Z}_5} g_{\{\varepsilon_i, \varepsilon_{i+1}\}}= 1 + \frac{3(\bar p + q)}{2} - \frac{1}{2}\sum \limits_{i\in \mathbb{Z}_5} g_{\{\varepsilon_i, \varepsilon_{i+1}\}}. 
\end{equation}
  
Now, by Proposition \ref{prop:preliminaries}(a), we know that $2g_{ijk}=g_{ij}+g_{ik}+g_{jk}-\frac{\nu(\Gamma)}{2}$ for any distinct $i,j,k \in \Delta_4$.
Thus, we have $g_{ij}+g_{ik}+g_{jk}=2g_{ijk}+(\bar p+q)$ for $0 \leq i <j<k\leq 4$. This gives ten linear equations which can be written in the following form.

$$ AX=B,$$

\noindent where

 $$A = \begin{bmatrix}
1 & 1 & 0 & 0 & 1 & 0 & 0 & 0 & 0 & 0\\
1 & 0 & 1 & 0 & 0 & 1 & 0 & 0 & 0 & 0\\
1 & 0 & 0 & 1 & 0 & 0 & 1 & 0 & 0 & 0\\
0 & 1 & 1 & 0 & 0 & 0 & 0 & 1 & 0 & 0\\
0 & 1 & 0 & 1 & 0 & 0 & 0 & 0 & 1 & 0\\
0 & 0 & 1 & 1 & 0 & 0 & 0 & 0 & 0 & 1\\
0 & 0 & 0 & 0 & 1 & 1 & 0 & 1 & 0 & 0\\
0 & 0 & 0 & 0 & 1 & 0 & 1 & 0 & 1 & 0\\
0 & 0 & 0 & 0 & 0 & 1 & 1 & 0 & 0 & 1\\
0 & 0 & 0 & 0 & 0 & 0 & 0 & 1 & 1 & 1\\
	\end{bmatrix}, X=
	\begin{bmatrix}
	g_{\{0,1\}}\\
	g_{\{0,2\}}\\
	g_{\{0,3\}}\\
	g_{\{0,4\}}\\
	g_{\{1,2\}}\\
	g_{\{1,3\}}\\
	g_{\{1,4\}}\\
	g_{\{2,3\}}\\
	g_{\{2,4\}}\\
	g_{\{3,4\}}\\
	\end{bmatrix} \mbox{and } B=
	\begin{bmatrix}
	2g_{\{0,1,2\}}+ \bar p+q\\
	2g_{\{0,1,3\}}+ \bar p+q\\
	2g_{\{0,1,4\}}+ \bar p+q\\
	2g_{\{0,2,3\}}+ \bar p+q\\
	2g_{\{0,2,4\}}+ \bar p+q\\
	2g_{\{0,3,4\}}+ \bar p+q\\
	2g_{\{1,2,3\}}+ \bar p+q\\
	2g_{\{1,2,4\}}+ \bar p+q\\
	2g_{\{1,3,4\}}+ \bar p+q\\
	2g_{\{2,3,4\}}+ \bar p+q\\
	\end{bmatrix}.
	$$
\noindent Therefore,

$$ X= A^{-1}B,$$

\noindent where
 $$A^{-1}=
 \begin{bmatrix}
1/3 & 1/3 & 1/3 & -1/6 & -1/6 & -1/6 & -1/6 & -1/6 & -1/6 & 1/3\\
1/3 & -1/6 & -1/6 & 1/3 & 1/3 & -1/6 & -1/6 & -1/6 & 1/3 & -1/6\\
-1/6 & 1/3 & -1/6 & 1/3 & -1/6 & 1/3 & -1/6 & 1/3 & -1/6 & -1/6\\
-1/6 & -1/6 & 1/3 & -1/6 & 1/3 & 1/3 & 1/3 & -1/6 & -1/6 & -1/6\\
1/3 & -1/6 & -1/6 & -1/6 & -1/6 & 1/3 & 1/3 & 1/3 & -1/6 & -1/6\\
-1/6 & 1/3 & -1/6 & -1/6 & 1/3 & -1/6 & 1/3 & -1/6 & 1/3 & -1/6\\
-1/6 & -1/6 & 1/3 & 1/3 & -1/6 & -1/6 & -1/6 & 1/3 & 1/3 & -1/6\\
-1/6 & -1/6 & 1/3 & 1/3 & -1/6 & -1/6 & 1/3 & -1/6 & -1/6 & 1/3\\
-1/6 & 1/3 & -1/6 & -1/6 & 1/3 & -1/6 & -1/6 & 1/3 & -1/6 & 1/3\\
1/3 & -1/6 & -1/6 & -1/6 & -1/6 & 1/3 & -1/6 & -1/6 & 1/3 & 1/3\\
	\end{bmatrix}.$$	
	
\noindent Since
$g_{\{i,j,k\}}=(m+1)+t_{\{i,j,k\}}$ for any distinct $i,j,k \in \Delta_4$, we consider
$$B = M+T,$$	

\noindent where

$$M= \begin{bmatrix}
	2(m+1)+ \bar p+q\\
	2(m+1)+ \bar p+q\\
	2(m+1)+ \bar p+q\\
	2(m+1)+ \bar p+q\\
	2(m+1)+ \bar p+q\\
	2(m+1)+ \bar p+q\\
	2(m+1)+ \bar p+q\\
	2(m+1)+ \bar p+q\\
	2(m+1)+ \bar p+q\\
	2(m+1)+ \bar p+q\\
	\end{bmatrix} \mbox{ and }
	T= \begin{bmatrix}
	2t_{\{0,1,2\}}\\
	2t_{\{0,1,3\}}\\
	2t_{\{0,1,4\}}\\
	2t_{\{0,2,3\}}\\
	2t_{\{0,2,4\}}\\
	2t_{\{0,3,4\}}\\
	2t_{\{1,2,3\}}\\
	2t_{\{1,2,4\}}\\
	2t_{\{1,3,4\}}\\
	2t_{\{2,3,4\}}\\
	\end{bmatrix}.$$
	
\noindent Thus,
$$X = A^{-1}M + A^{-1} T.$$
Therefore,
$g_{\{i,j\}}=\frac{2(m+1)+\bar p+q}{3} + 2\sum_{0 \leq k <l<r\leq 4}c_{klr}^{ij}t_{\{k,l,r\}}$, where $c_{klr}^{ij}$ is the element of $A^{-1}$ corresponding to
$\{i,j\}$-row and $\{k,l,r\}$-column of $A^{-1}$. Thus for a given permutation 
$\varepsilon= (\varepsilon_0, \varepsilon_1, \varepsilon_2, \varepsilon_3, \varepsilon_4=4)$ of $\Delta_4$, we have 

$$\sum_{i \in \mathbb{Z}_5} g_{\{\varepsilon_i, \varepsilon_{i+1}\}}=\frac{5}{3}(2(m+1)+\bar p+q)+ 2\sum_{0 \leq k <l<r\leq 4}(\sum_{i \in \mathbb{Z}_5}  c_{klr}^{\varepsilon_i \varepsilon_{i+1}})t_{\{k,l,r\}}.$$ 

Therefore, from Equation \eqref{eq:rho} we get
\begin{align*}
\rho_{\varepsilon} (\Gamma) &= 1 + \frac{3(\bar p + q)}{2} - \frac{5}{6}(2(m+1)+\bar p+q)- \sum_{0 \leq k <l<r\leq 4}(\sum_{i \in \mathbb{Z}_5}  c_{klr}^{\varepsilon_i \varepsilon_{i+1}})t_{\{k,l,r\}}\\
&=1+\frac{2\bar p - 5(m+1)}{3} +\frac{2q}{3}-\sum_{0 \leq k <l<r\leq 4}(\sum_{i \in \mathbb{Z}_5}  c_{klr}^{\varepsilon_i \varepsilon_{i+1}})t_{\{k,l,r\}} \\
&= 2 \chi (M) + 5m -4 +\frac{2}{3 }\sum_{0 \leq k <l<r\leq 4}t_{\{k,l,r\}}-\sum_{0 \leq k <l<r\leq 4}(\sum_{i \in \mathbb{Z}_5}  c_{klr}^{\varepsilon_i \varepsilon_{i+1}})t_{\{k,l,r\}}.
\end{align*}

Now, for a given permutation $\varepsilon= (\varepsilon_0, \varepsilon_1, \varepsilon_2, \varepsilon_3, \varepsilon_4=4)$ of $\Delta_4$, we have the following values of $\sum_{i \in \mathbb{Z}_5}c_{klr}^{\varepsilon_i\varepsilon_{i+1}}$ which is a coefficient of $t_{\{k,l,r\}}$ in $\sum_{0 \leq k <l<r\leq 4}(\sum_{i \in \mathbb{Z}_5}  c_{klr}^{\varepsilon_i \varepsilon_{i+1}})t_{\{k,l,r\}}$.

 \begin{center}
\begin{tabular}{|c|c|c|c|c|c|c|c|c|c|c|}
\hline
$\sum_{i \in \mathbb{Z}_5}c_{klr}^{\varepsilon_i\varepsilon_{i+1}}$, $klr=$ & $012$ & $013$ & $014$ & $023$ & $024$ & $034$ & $123$ & $124$ & $134$ & $234$ \\
\hline
\hline
$\varepsilon=(0,1,2,3,4)$ & 2/3 & -1/3 & 2/3 & -1/3 & -1/3 & 2/3 & 2/3 & -1/3 & -1/3 & 2/3 \\
\hline
$\varepsilon=(0,1,3,2,4)$ & -1/3 & 2/3 & 2/3 & -1/3 & 2/3 & -1/3 & 2/3 & -1/3 & -1/3 & 2/3 \\
\hline
$\varepsilon=(0,2,3,1,4)$ & -1/3 & -1/3 & 2/3 & 2/3 & 2/3 & -1/3 & 2/3 & -1/3 & 2/3 & -1/3 \\
\hline
$\varepsilon=(0,2,1,3,4)$ & 2/3 & -1/3 & -1/3 & -1/3 & 2/3 & 2/3 & 2/3 & -1/3 & 2/3 & -1/3 \\
\hline
$\varepsilon=(0,3,1,2,4)$ & -1/3 & 2/3 & -1/3 & -1/3 & 2/3 & 2/3 & 2/3 & 2/3 & -1/3 & -1/3 \\
\hline
$\varepsilon=(0,3,2,1,4)$ & -1/3 & -1/3 & 2/3 & 2/3 & -1/3 & 2/3 & 2/3 & 2/3 & -1/3 & -1/3 \\
\hline
$\varepsilon=(1,3,0,2,4)$ & -1/3 & 2/3 & -1/3 & 2/3 & 2/3 & -1/3 & -1/3 & 2/3 & 2/3 & -1/3 \\
\hline
$\varepsilon=(1,0,3,2,4)$ & -1/3 & 2/3 & 2/3 & 2/3 & -1/3 & -1/3 & -1/3 & 2/3 & -1/3 & 2/3 \\
\hline
$\varepsilon=(1,0,2,3,4)$ & 2/3 & -1/3 & 2/3 & 2/3 & -1/3 & -1/3 & -1/3 & -1/3 & 2/3 & 2/3 \\
\hline
$\varepsilon=(1,2,0,3,4)$ & 2/3 & -1/3 & -1/3 & 2/3 & -1/3 & 2/3 & -1/3 & 2/3 & 2/3 & -1/3 \\
\hline
$\varepsilon=(2,1,0,3,4)$ & 2/3 & 2/3 & -1/3 & -1/3 & -1/3 & 2/3 & -1/3 & 2/3 & -1/3 & 2/3 \\
\hline
$\varepsilon=(2,0,1,3,4)$ & 2/3 & 2/3 & -1/3 & -1/3 & 2/3 & -1/3 & -1/3 & -1/3 & 2/3 & 2/3 \\
\hline
\end{tabular}
 \end{center} 

\noindent Thus, $$\sum_{0 \leq k <l<r\leq 4}(\sum_{i \in \mathbb{Z}_5}  c_{klr}^{\varepsilon_i \varepsilon_{i+1}})t_{\{k,l,r\}}= \frac{2}{3}\sum_{i \in \mathbb{Z}_5} t_{\{\varepsilon_i,\varepsilon_{i+1},\varepsilon_{i+2}\}}-\frac{1}{3} \sum_{i \in \mathbb{Z}_5} t_{\{\varepsilon_i,\varepsilon_{i+2},\varepsilon_{i+4}\}}.$$
 
\noindent Since $\sum_{0 \leq k <l<r\leq 4}t_{\{k,l,r\}}= \sum_{i \in \mathbb{Z}_5} t_{\{\varepsilon_i,\varepsilon_{i+1},\varepsilon_{i+2}\}}+ \sum_{i \in \mathbb{Z}_5} t_{\{\varepsilon_i,\varepsilon_{i+2},\varepsilon_{i+4}\}}$, we have

\begin{align*}
\rho_{\varepsilon} (\Gamma) &= 2 \chi (M) + 5m -4 +\sum_{i \in \mathbb{Z}_5} t_{\{\varepsilon_i,\varepsilon_{i+2},\varepsilon_{i+4}\}}.
\end{align*}

Now, the result follows from this.
\end{proof}

\medskip

\noindent {\em Proof of Theorem}  \ref{teorem:main}.
Let $M$ admit a weak semi-simple crystallization. Thus, we have $g_{\{i,i+1,i+2\}}=m+1$ for $i \in \Delta_4$ (addition in subscript of $g$ is modulo 5).
Therefore, from Lemma \ref{lemma:main1}, we have $\rho_{\varepsilon} (\Gamma) = 2 \chi (M) + 5m -4 +\sum_{i \in \mathbb{Z}_5} (g_{\{\varepsilon_i,\varepsilon_{i+2},\varepsilon_{i+4}\}}-m-1)$ for any cyclic permutation $\varepsilon=(\varepsilon_0,\dots,\varepsilon_4=4)$ of the color set $\Delta_4$. Let $\varepsilon=(2,0,3,1,4)$. Then  $g_{\{\varepsilon_i,\varepsilon_{i+2},\varepsilon_{i+4}\}}=g_{\{\varepsilon_i,\varepsilon_{i}+1,\varepsilon_{i}+2\}}=m+1$ (addition in subscript of $g$ is modulo 5). Therefore, $\rho_{\varepsilon} (\Gamma) = 2 \chi (M) + 5m -4$ and hence, $\mathcal{G}(M) \leq 2 \chi (M) + 5m -4$. Thus, by Proposition \ref{prop:lowerbound}, $\mathcal{G}(M) = 2 \chi (M) + 5m -4$.

On the other hand, let $\mathcal G(M)=  2 \chi (M) + 5m -4$. Then, from Lemma \ref{lemma:main1} we have,
there exist a crystallization $\Gamma$ and a cyclic permutation $\varepsilon=(\varepsilon_0,\dots,\varepsilon_4=4)$ of the color set $\Delta_4$ such that $\sum_{i \in \mathbb{Z}_5} (g_{\{\varepsilon_i,\varepsilon_{i+2},\varepsilon_{i+4}\}}-m-1)=0$. Therefore, $g_{\{\varepsilon_i,\varepsilon_{i+2},\varepsilon_{i+4}\}}=m+1$ (addition in subscript of $g$ is modulo 5). Let $\bar \Gamma$ be a colored graph, obtained from $\Gamma$, by replacing the colors $(\varepsilon_0,\varepsilon_2,\varepsilon_4,\varepsilon_1,\varepsilon_3)$ by $(0,1,2,3,4)$. Then $\bar{\Gamma}$ is a crystallization of $M$ with $\bar g_{\{i,i+1,i+2\}}=m+1$ for $i \in \mathbb{Z}_5$. Therefore, $M$ admits a weak semi-simple crystallization. \hfill $\Box$

\begin{corollary}\label{cor:min-genus}
If a closed connected PL $4$-manifold $M$ admits a crystallization $(\Gamma,\gamma)$ such that $\nu(\Gamma)\leq 6 \chi(M) + 20 m - 6$, where $m$ is the rank of the fundamental group of $M$ then $\mathcal G(M) =\rho(\Gamma)=  2 \chi (M) + 5m -4$. 
\end{corollary}
\begin{proof}
From Lemma \ref{lemma:vertex}, we have $ 6 \chi(M) +2\sum_{0 \leq i <j<k\leq 4}g_{\{i,j,k\}}-30 = \nu(\Gamma) \leq 6 \chi(M) + 20 m - 6$. Therefore, $\sum_{0 \leq i <j<k\leq 4}g_{\{i,j,k\}} \leq 10 m + 12$. On the other hand, by Proposition \ref{prop:preliminaries}(d), we have $g_{\{i,j,k\}}\geq m+1$. Thus, at most two $g_{\{i,j,k\}}>m+1$. Therefore,
there exists a cyclic permutation $(\delta_0,\dots,\delta_4=4)$ of the color set $\Delta_4$ such that $g_{\{\delta_i,\delta_{i+1},\delta_{i+2}\}}=m+1$ for $i \in \Delta_4$ (addition in subscript of $g$ is modulo 5). Let $\bar \Gamma$ be a colored graph, obtained from $\Gamma$, by replacing the colors $\delta_i$ by $i$ for $i \in \mathbb{Z}_5$. Then $\bar{\Gamma}$ is a crystallization of $M$ with $\bar g_{\{i,i+1,i+2\}}=m+1$ for $i \in \mathbb{Z}_5$. Therefore, $M$ admits a weak semi-simple crystallization. Now, the result follows from Theorem \ref{teorem:main}.
\end{proof}

\begin{corollary}\label{corollary:novik-swartz}
Let $\mathbb{F}$ be a field, let $(\Gamma,\gamma)$ be a crystallization of a  closed $\mathbb{F}$-orientable PL $4$-manifold $M$ and, let $\mathcal{T}$ be the corresponding simplicial cell-complex. If  $\mathcal{T}$ satisfies equality in the Novik-Swartz bound with respect to $\mathbb{F}$ then $\mathcal{G}(M)=\rho(\Gamma)=2 \chi (M) + 5m -4$.
 \end{corollary}

\begin{proof}
Set $f_4 (\mathcal{T}) = 2p$, and let 
$h (\mathcal{T}) = (h_0(\mathcal{T}), \ldots ,	h_{5}(\mathcal{T}))$
be the $h$-vector of $\mathcal{T}$, i.e., 
\begin{equation}\label{eq:h-vector}
\sum \limits_{i=0}^{5} h_i(\mathcal{T}) x^{5-i} = 
\sum_{i=0}^5 f_{i-1}(\mathcal{T})(x-1)^{5-i}.
\end{equation}
Then, due to Novik and Swartz 
(cf. \cite[Theorem 6.4]{ns09}), we have 
\begin{equation}\label{novikswartz}
  h_j(X) \geq \binom{5}{j} \sum_{i=0}^{j}(-1)^{j-i} 
  \beta_{i-1}(M,\mathbb{Z}).
\end{equation}
By substituting $x$ by $x+1$ and comparing the coefficients of $x^i$ on 
both sides of Equation \eqref{eq:h-vector}, we obtain $2p = \sum_{i=0}^5 h_{i}(X)$.
Thus, Equation~(\ref{novikswartz}) and the Euler characteristic imply
\begin{eqnarray}\label{novik}
 \begin{array}{lllllllllll}
    \beta_0(M,\mathbb{F}) &+&  4\beta_1(M,\mathbb{F}) &+& 
    6\beta_2(M,\mathbb{F})  &+&  4\beta_3(M,\mathbb{F})  &+& 
    \beta_4(M,\mathbb{F}) &\leq& 2p, \\
    \beta_0(M,\mathbb{F}) & - & \beta_1(M,\mathbb{F}) & + & 
    \beta_2(M,\mathbb{F}) & - & \beta_3(M,\mathbb{F}) & + & 
    \beta_4(M,\mathbb{F}) &=& \chi(M).
\end{array} 
\end{eqnarray}
Since we have equality in the both equations in \eqref{novik}, it follows that
$$ \begin{array}{lllllllllll}
  5\beta_0(M,\mathbb{F}) & - & 10\beta_1(M,\mathbb{F}) & - & 
  10\beta_3(M,\mathbb{F}) & + & 5\beta_4(M,\mathbb{F}) &=& 
  6\chi(M) - 2p.
\end{array} $$
Let $m$ be the rank of the fundamental group of $M$. Then $m \geq 
\beta_1(M,\mathbb{F}) = \beta_3(M,\mathbb{F})$. Thus, $2p \leq 6 
\chi(M) + 10(2m-1)$. Now, the result follows from Corollary \ref{cor:min-genus}.
\end{proof}

\begin{lemma}\label{Theorem:additive}
Let $M_1$ and $M_2$ be two PL $4$-manifolds admitting weak semi-simple crystallizations. Then $M_1\# M_2$ admits a weak semi-simple crystallization.
\end{lemma}

\begin{proof}
For $1\leq i \leq 2$, let $m_i$ be the rank of the fundamental group of $M_i$. Since $M_i$ admits a weak semi-simple crystallization, there exists a crystallization $(\Gamma^i,\gamma^i)$ with color set $\Delta_4$ such that $g^i_{\{j,j+1,j+2\}}=m_i+1$ for $j \in \Delta_4$ (addition in subscript of $g^i$ is modulo 5) and $1\leq i \leq 2$. Let $\bar \Gamma=\Gamma^1 \#_{uv} \Gamma^2$. Then $\bar g_{\{j,j+1,j+2\}}=m_1+m_2+1$ for $j \in \Delta_4$ (addition in subscript of $\bar g$ is modulo 5).
Now, by Proposition \ref{prop:preliminaries}(c), $\bar \Gamma=\Gamma^1 \#_{uv} \Gamma^2$ is a crystallization of $M_1 \# M_2$ for some $u\in V(\Gamma^1)$ and $ v \in V(\Gamma^2)$. Therefore, $M_1\# M_2$ admits a weak semi-simple crystallization.
\end{proof}

Now, by using Theorem \ref{teorem:main}, we have the following.

\begin{corollary}\label{corollary:genus}
For  $i\leq i \leq 2$, if $\mathcal G(M_i)=2 \chi (M_i) + 5m_i -4$, where $m_i$ is the rank of the fundamental group of $M_i$, then $\mathcal{G}(M_1\# M_2) = 2 \chi (M_1 \# M_2) + 5(m_1+m_2)-4$.
\end{corollary}

\medskip

\noindent {\em Proof of Theorem}  \ref{theorem:addative}.
For  $1\leq i \leq 2$, let $m_i$ be the rank of the fundamental group of $M_i$. Thus, $\mathcal{G}(M_1\# M_2) = 2 \chi (M_1 \# M_2) + 5(m_1+m_2)-4$. Since $\chi (M_1 \# M_2)=\chi (M_1)+\chi (M_2)-2$, we have $\mathcal{G}(M_1\# M_2) = (2 \chi (M_1) + 5m_1-4)+ (2 \chi (M_2) + 5m_2-4)=\mathcal{G}(M_1)+\mathcal{G}(M_2)$. \hfill $\Box$

\begin{remark}\label{rem:huge-class}
{\rm
From Definition \ref{def:weak}, it is clear that, if a PL 4-manifold $M$ admits a semi simple crystallization (in particular, simple crystallization) then $M$ admits a weak semi simple crystallization. Therefore, from Theorem \ref{teorem:main}, we have
$\mathcal{G}(M) = 2 \chi (M) + 5$ rk$(\pi_1(M))-4$. From \cite{bc15,bj16}, we know  the class of PL 4-manifolds admitting a weak semi simple crystallization, contains  $S^4$, $\mathbb{CP}^{2}$, $S^{2} \times S^{2}$,  $\mathbb{RP}^4,$  the orientable and non-orientable $S^3$-bundles over $S^1$ and the $K3$-surface, together with their connected sums possibly with reverse orientation also.
Moreover,

\begin{itemize}
\item from \cite[Figure 3]{bc15}, we know  $\mathcal{G}(\mathbb{RP}^2\times S^2)=5=2\cdot 2+5 \cdot 1-4= 2 \chi (\mathbb{RP}^2\times S^2) + 5$ rk$(\pi_1(\mathbb{RP}^2\times S^2))-4$, and hence $\mathbb{RP}^2\times S^2$ admits a weak semi-simple crystallization.

\item from \cite[Theorem 1]{ba16}, we know that there exists an orientable (resp., a non-orientable) mapping torus $(S^2\times S^1)_f$ (with fundamental group $\mathbb{Z}\times \mathbb{Z}$) of a map $f:S^2\times S^1 \to S^2\times S^1$ such that  $\mathcal{G}((S^2\times S^1)_f)= 6=2\cdot 0+5 \cdot 2 -4= 2 \chi ((S^2\times S^1)_f) + 5$ rk$(\pi_1(S^2\times S^1)_f))-4$. Therefore, $(S^2\times S^1)_f$ admits a weak semi-simple crystallization.
\end{itemize}
}
\end{remark}

\begin{remark}\label{rem:conjecture}
{\rm
Let $M$ be a closed connected PL $4$-manifold admitting a semi-simple crystallization $(\Gamma,\gamma)$. Then, regular genus of $M$ depends on the number of vertices of $\Gamma$ and we have only finitely many for such PL $4$-manifolds having same regular genus. Thus, we have introduced the notion weak semi-simple crystallization. The regular genus of closed connected PL $4$-manifold admitting a weak semi-simple crystallization does not depend on the number of vertices of its crystallizations. Thus, it may be possible to find infinite number of such PL 4-manifolds having same regular genus. 

Finally, we note that all PL 4-manifolds whose regular genus is known, admit a weak semi-simple crystallization. Thus, we believe that all closed connected (simply-connected) PL $4$-manifolds admit a weak semi-simple crystallization.
Therefore, in virtue of Theorem \ref{theorem:addative}, we believe that the conjecture ``additivity of regular genus under connected sum for closed connected PL $d$-manifold" is true for $d=4$.

}
\end{remark}

\begin{remark}{\rm
Lemma \ref{lemma:main1} has been used only to prove Theorem \ref{teorem:main}. There is a minor (printing) error in the statement (and in the second last line of the proof) of Lemma \ref{lemma:main1} in the published version \cite{ba17} of this article: the value of $\rho_{\varepsilon} (\Gamma)$ was $2 \chi (M) + 5m -4 +\frac{1}{3 }\sum_{i \in \mathbb{Z}_5} (g_{\{\varepsilon_i,\varepsilon_{i+2},\varepsilon_{i+4}\}}-m-1)$ in the published version \cite{ba17}. The corrected value of $\rho_{\varepsilon} (\Gamma)$ is $2 \chi (M) + 5m -4 +\sum_{i \in \mathbb{Z}_5} (g_{\{\varepsilon_i,\varepsilon_{i+2},\varepsilon_{i+4}\}}-m-1)$. However, this minor (printing) error does not effect the proof (and hence the statement) of Theorem \ref{teorem:main}. The author would like to thank Maria Rita Casali for pointing out this error.
}\end{remark}

\noindent {\bf Acknowledgements:}
The author would like to thank Basudeb Datta for some useful comments and suggestions. The author thanks Jonathan Spreer for the idea of the proof of Corollary \ref{corollary:novik-swartz}. The major portion of this work has been done at ISI Bangalore, India under NBHM  Postdoctoral fellowship. The author is supported by SERB Indo-US Postdoctoral Fellowship (2016/ 110-Biplab Basak).

\end{document}